\documentclass[12pt, reqno]{amsart}

\usepackage{amsfonts,amssymb}
\usepackage{hyperref, graphicx}
\usepackage{epsfig}
\usepackage{latexsym}
\usepackage{amsmath,amsthm,calligra,mathrsfs}
\usepackage{mathrsfs}
\usepackage[all,cmtip]{xy}
\usepackage{comment}
\usepackage[dvipsnames]{xcolor}
\usepackage{tikz-cd}
\usepackage{mathrsfs}
\usepackage{bbm}
\theoremstyle{plain}
\newtheorem{theorem}{Theorem}[section]
\newtheorem{lemma}[theorem]{Lemma}
\newtheorem{definition}[theorem]{Definition}
\newtheorem{proposition}[theorem]{Proposition}
\newtheorem{cor}[theorem]{Corollary}
\newtheorem{remark}[theorem]{Remark}
\newtheorem{example}[theorem]{Example}

\numberwithin{equation}{section}

\def \d{\mathbf{d}}
\def \sd{\mathbf{d^s}}
\def \sdp{\mathbf{d}'^{\mathbf{s}}}
\def \N{\mathbb{N}}
\def \Z{\mathbb{Z}}

\def \ra{\rightarrow}
\def \bS{\mathbb{S}}
\def \bfR{\mathbf{Rep}}
\def \bfSR{\mathbf{SR}}
\def \btheta{\boldsymbol{\theta}}
\def \bbSR{\mathbb{S}\mathbf{R}}
\def \G{\mathbb{G}}
\def \X{\mathbb{X}}

\newcommand{\Hom}{\operatorname{Hom}\nolimits}
\newcommand{\Sym}{\operatorname{Sym}\nolimits}
\newcommand{\sheaf}[1]{\mathcal{#1}}
\newcommand{\sdim}{\mathrm{sdim}}
\newcommand{\salg} {\textbf{ k-salg }}
\newcommand{\grps} {\textbf{Grps}}
\newcommand{\smod} {k\text{-}\textbf{smod }}
\newcommand\spec{{\mathrm{Spec}}}
\newcommand{\uHom}{\mathrm{\underline{Hom}}}

\newcommand{\bp} {\mathbf {p}}
\newcommand{\bA} {\mathbb{A}}
\newcommand{\rM}{\mathrm{M}}
\newcommand{\GL}[1][]{\mathrm{GL_{p|q}}{#1}}
\newcommand{\GLV}[1][]{\mathrm{GL_{V}}{#1}}
\newcommand{\Gm}{\mathbb{G}_m}
\newcommand{\mc}[1]{\mathcal{#1}}
\newcommand{\stsheaf}[1]{\mathcal{O}_{#1}}
\newcommand{\est}[1]{\mathcal{O}_{{#1}_0}}
\newcommand{\ost}[1]{\mathcal{O}_{{#1}_1}}
\newcommand{\Spec}[1]{\mathbb{S}\mathrm{pec}(#1)}
\newcommand{\Proj}[1]{\mathbb{P}\mathrm{roj}\left(#1\right)}
\newcommand{\smfss}{\mathfrak{M}_{\sd}^{\btheta\text{-}\mathrm{ss}}}
\newcommand{\smfs}{\mathfrak{M}_{\sd}^{\btheta\text{-}\mathrm{st}}}
\newcommand{\smss}{\mathsf{M}_{\sd}^{\btheta\text{-}\mathrm{ss}}}
\newcommand{\sms}{\mathsf{M}_{\sd}^{\btheta\text{-}\mathrm{st}}}
\newcommand{\SSch}{\mathbf{SSch}}
\newcommand{\Sch}{\mathbf{Sch}}
\newcommand{\Set}{\mathbf{Set}}
\newcommand{\smf}{\mathfrak{M}_{\sd}}
\newcommand{\Ms}{\mathcal{M}}
\newcommand{\fp}[1]{\underline{#1}}
\title{Moduli of super-representations of quivers}
\author{Sanjay~Amrutiya}
\address{Department of Mathematics, IIT Gandhinagar,
 Near Village Palaj, Gandhinagar - 382355, India}
 \email{samrutiya@iitgn.ac.in}
\author{Umesh ~V ~Dubey}
\address{Harish-Chandra Research Institute,
Prayagraj - 211019, Uttar Pradesh, India} 
\address{Homi Bhabha National Institute, Training School Complex, Anushakti Nagar, Mumbai 400 094, India}
 \email{umeshdubey@hri.res.in}
\date{}

\begin{document}
\maketitle
\begin{abstract}
 In this note, we construct moduli spaces of super-representations of quivers by extending King’s Geometric Invariant Theory (GIT) framework to the super setting. 
 Using the even part of super general linear groups and a parity-shifting operator, we construct moduli spaces that topologically parameterize super-representations of a fixed super-dimension. 

 We also outline the construction of super quiver varieties for super-representations of framed quivers (without doubling); and realise super-Grassmannians and super-flag varieties as geometric quotients within this setting.
\end{abstract}
\section{Introduction}
A. King \cite{Ki94} introduced geometric methods to study the 
representations of quivers, using Mumford's GIT. Since then, 
the quiver varieties and their applications have been extensively 
explored by several mathematicians. Recently, Bovdi and Zubkov
\cite{BZ20} have initiated the study of super-representations 
of quiver and their semi-invariants with respect to the super 
general linear groups. This motivated us to study super-representations of quivers using geometric methods.

In this article, we have begun to study the geometric 
quotients of super vector spaces by the action of standard 
linearly reductive groups. Note that there are very few  
actual linearly reductive super-groups. In fact, the super 
general linear groups are not linearly reductive, in 
general \cite{S21}. 

Let us briefly outline the content of the article. In Section 2, we recall various fundamental concepts from super algebraic geometry. Section 3 addresses the action of a standard linearly reductive group on an affine superscheme and describes the construction of super GIT quotients. In Section 4, we first review some basics of super representations of quivers and then explore the algebraic notion of stability for super representations, relating it to the usual representations of the associated quiver. In Section 5, we consider the super representation space with the action of the even part of the super general linear groups to construct a moduli space of super representations of a quiver with a fixed super dimension. To topologically parameterize super representations with a fixed super dimension, we need to double the super representation space using a parity-shifting operator. Finally, in the last section, we discuss the super GIT quotient (without doubling) to construct super quiver varieties, which parametrize the super representations of quivers within the category of super vector spaces. As a special case, we can obtain super projective spaces, super Grassmannians, and super flag varieties.

\section{Preliminaries}
In this section, we briefly review some basic concepts in algebraic super geometry. For more details see \cite{BR21, BRP, CN, CCF, Ma88a, MZ22, OV23, We09}.

\subsection{Super spaces and bosonic reduction}
A super space $X$ is a locally super ringed space $(|X|, \stsheaf{X})$, 
i.e., a pair consisting of an underlying  topological space $|X|$ and a structure 
sheaf 
$$\stsheaf{X} = \est{X} \oplus \ost{X}$$
of supercommutative rings such that the stalks are local 
supercommutative rings. The component
$\est{X}$ (respectively, $\ost{X}$) is called the sheaf 
of even (resp. odd) functions.

Let $$\mc{J} = \ost{X} \oplus \ost{X}^2$$ for
the ideal in the structure sheaf $\stsheaf{X}$ generated by 
$\ost{X}$. The ringed space 
$$X_{\mathrm{bos}} := (|X|, \stsheaf{X}/\mc{J}),$$
which is a purely even subspace of $X$, is called the bosonic 
reduction of $X$. 
There is a canonical closed embedding $X_{\mathrm{bos}} \hookrightarrow X$ induced by this quotient map of structure sheaves. It is a final object among all locally (even)  ringed spaces mapping (degree preserving) to $X$. 

If the closed immersion 
induced by $\stsheaf{X}\rightarrow \stsheaf{X}/\mathcal{J}$ has 
a retraction $\pi : X \rightarrow X_\mathrm{bos}$, then we say 
that $X$ is projected.

The reduced ringed space associated with the super space 
$(|X|, \stsheaf{X})$ s defined as the pair: 
$$X_{\mathrm{red}} := (|X|, \stsheaf{X}/\mc{N}),$$
where $\mc{N}$ is the nilradical ideal sheaf of $\stsheaf{X}$. In general if $X$ is a superscheme over a commutative ring $k$ and $2$ is not a zero-divisor in $k$, we have a closed immersion
$$X_{\mathrm{red}} \hookrightarrow X_{\mathrm{bos}}$$

\subsection{Affine super schemes}
\textit{We assume that all superalgebras are associative, super-commutative 
(i.e. $xy=(-1)^{p(x)p(y)} yx$) with unit and over a commutative ring $k$ and $2$ is not a zero-divisor in $k$.
We denote the category of such all superalgebras by} $\salg$.

For a supercommutative ring $A$, the standard construction of an affine scheme generalizes easily to give an affine superscheme 
$\Spec{A}$.

For a superalgebra $A$ let
$J_A$ denote the ideal generated by the odd elements i.e. $J_A= \langle A_1 \rangle_A$.
Denote the quotient $A/J_A$ by $A_\mathrm{bos}$.

Let $A$ be an object of $\salg$. We have that the underlying topological spaces 
$|\spec(A_0)|=|\spec(A_\mathrm{bos})|$, since the
algebras $A_\mathrm{bos}$ and $A_0$ differ only by nilpotent elements by our hypothesis. 

Let's consider $\mc{O}_{A_0}$ the structural sheaf of $\spec(A_0)$.
The stalk of the sheaf at the prime $\bp\in
\spec(A_0)$ is the localization of $A_0$ at $\bp$.
As for any superalgebra, $A$ is a
module over $A_0$. We have indeed a sheaf $\mc{O}_A$ of
$\mc{O}_{A_0}$-modules over $\spec A_0$ with stalk $A_{\bp}$, the localization
of the $A_0$-module $A$ over each prime
$\bp \in \spec(A_0)$.
$$
A_{\bp}=\{ {\frac{f}{g}} \quad | \quad f\in A, g \in A_0-\bp\}
$$
The localization $A_\bp$ has a unique two-sided maximal ideal 
which consists of
the maximal ideal in the local ring
$(A_\bp)_0$ and the generators of $(A_\bp)_1$ as $A_0$-module.
This uniqely determine the sheaf $\mc{O}_A$ of superalgebras on $\spec A_0$ such that for a basic open set
$$
U_f=\{ \bp \in \spec A_0 | (f) \not\subset \bp \}, \qquad f \in A_0
$$
we have $\mc{O}_A(U_f) := A_f=\{a/f^n \quad | \quad a \in A\}$.
The superspace $(|\spec A_0|,\mc{O}_A)$ is called \textit{affine superscheme}, and is denoted by $\Spec A$.

\begin{definition}
An \textit{ affine  algebraic supervariety} is a superspace
isomorphic to $\Spec A$ for some \textit{ affine}
superalgebra $A$ that is a finitely generated superalgebra such
that $A/J_A$ has no nilpotents.
We will call
$A$ the \textit{ coordinate ring} of the supervariety.
\end{definition}

\subsection{Super-scheme}
A \emph{superscheme} $X := (|X|, \mc{O}_X)$ is a 
superspace $(|X|, \mc{O}_X)$ such that $(|X|, \est{X})$ is a 
scheme (over $k$) and $\ost{X}$ is a quasi-coherent sheaf 
of $\est{X}$-modules. 

A \emph{morphism of superschemes} 
$f:X \to Y$ is a continuous map 
$f:|X| \to |Y|$ of the underlying topological spaces along with a morphism 
$f^{\#}: \mc{O}_Y \to f_* \mc{O}_X$ of sheaves of supercommutative 
$k$-algebras that is local, \emph{i.e}., induces a local morphism 
between the stalks.

A superspace $X$ is a superscheme if and only if it is locally
isomorphic to $\Spec A$ for some superalgebra $A$, i.e.
for all $x \in X$, there exists $U_x \subset X$ open such
that $(U_x, \stsheaf{X}|_{U_x}) \cong \Spec A$. 
Let $\SSch$ be the category of super-schemes.

\subsubsection{Affine Superspace}
Given a super vector space $V$ of dimension $m|n$ over the 
field $k$ of characteristic $\neq 2$, we define the corresponding \emph{affine superspace} as
\[
\mathbb{V} := \Spec{ \Sym (V^*)}.
\]
Recall that
$V^*$ is the set of linear maps $V \ra k$ not necessarily preserving
the parity, and $\Sym(V^*)=\Sym(V^*_0) \otimes \wedge V^*_1$,
where $\wedge V^*_1$ denotes the exterior algebra over the
ordinary vector space $V_1$. 
The standard \emph{affine $(m|n)$-superspace} is the affine 
superspace $\mathbb{A}^{m|n}$ associated with the vector 
superspace $V = k^{m|n}$. If we choose coordinates $x_1, \dots, 
x_m, \xi_1, \dots, \xi_n \in V^*$ on $V$, then $V$ may be 
identified with $k^{m|n}$ and
\[ \mathbb{V} \cong \mathbb{A}^{m|n} := \Spec{k[x_1, \cdots, x_m \; 
, \;  \xi_1, \cdots, \xi_n]}. \]

As a topological space
$\spec k[x_1 \dots x_m, \xi_1 \dots \xi_n]_0$ will consists of
the even maximal ideals
$$
(x_i-a_i, \xi_j\xi_k),
\qquad i=1 \dots m, \quad j,k=1 \dots n
$$
and the even prime ideals
$$
(p_1, \dots, p_r, \xi_j\xi_k),
\qquad i=1 \dots r, \quad j,k=1 \dots n
$$
where $(p_1, \dots, p_r)$ is a prime ideal in $k[x_1 \dots x_m]$.

The structural sheaf of $\bA^{m|n}$ will have stalk at the
point $\bp \in \spec k[\bA^{m|n}]_0$:
$$
k[\bA^{m|n}]_{\bp}=\{{\frac{f}{g}} \quad | \quad  f \in k[\bA^{m|n}],\quad
g\in k[\bA^{m|n}]_0, \quad  g\notin \bp\}.
$$

The bosonic reduction of $\mathbb{A}^{m|n}$ is the affine space $
\mathbb{A}^m$. The super affine space is split and may be regarded 
as a super ``ringification" of its bosonic reduction, the affine 
space $\mathbb{A}^m$, with a structure sheaf 
\[ \mc{O}_{\mathbb{A}^{m|n}} = \Sym_{\mc{O}_{\mathbb{A}^m}} (\Pi 
\mc{O}_{\mathbb{A}^m}^n),
\]
where $\Pi$ denotes the parity change operation, which shifts the 
grading by 1 modulo 2. The supersymmetric algebra of a purely odd 
linear object, in this case the sheaf 
$\Pi \mc{O}_{\mathbb{A}^m}^n$ of $\mc{O}_{\mathbb{A}^m}$-modules, 
is well-known as the exterior (Grassmann) algebra of the 
corresponding purely even object:
\[
\Sym (\Pi \mc{O}_{\mathbb{A}^m}^n) = 
\bigwedge(\mc{O}_{\mathbb{A}^m}^n).
\]

\begin{remark}
    If we take the supervector space $V \oplus \Pi V$ where $V$ has dimension $m|n$, then the bosonic reduction of the affine space associated to this supervector space will be $\mathbb{A}^m \times \mathbb{A}^n \cong \mathbb{A}^{m+n}$.
\end{remark}
\subsubsection{Projective Superspace}
Given a super vector space $V = V_0 \oplus V_1$ of dimension 
$(m+1)|n$, the $(m|n)$-dimensional super projective space 
$\mathbb{P}(V)$ may be defined as the superspace of lines, 
\emph{i.e}., $(1|0)$-dimensional vector subspaces of $V$. More 
technically, $\mathbb{P}(V)$ is the superscheme representing the 
functor of points that assigns a superscheme $T$ the set of 
\emph{locally supplemented} $(1|0)$-line subbundles is a locally free subsheaf of the trivial super vector bundle $T \times V$ 
over $T$. These are linear duals $\mc{L}^\vee$ of locally free, 
rank-$(1|0)$ quotient sheaves $\mc{L}$ of $\mc{O}_T \otimes 
V^\vee$ on $T$. One may also think of $\mathbb{P}(V)$ as the 
quotient
\[
\mathbb{P}(V) = (\mathbb{V}\setminus \{0\})/\mathbb{G}_m
\]
of the super affine space $\mathbb{V}$ with deleted origin by the 
\emph{multiplicative group} $\mathbb{G}_m = \mathrm{GL}(1)$, 
acting on $\mathbb{V}\setminus \{0\}$ by dilations. Finally, the 
construction of a projective spectrum generalizes to the super 
case, and one may identify
\[ \mathbb{P}(V) := \Proj{\Sym(V^\vee)}, \]
where the algebra $\Sym(V^\vee)$ is 
$\mathbb{Z} \oplus \mathbb{Z}_2$-graded with the $\mathbb{Z}$-
grading coming from the symmetric powers and the $\mathbb{Z}_2$-
grading coming from the $\mathbb{Z}_2$-grading on $V^\vee$. 
It turns out the super projective space over $k$ is isomorphic to 
a split superscheme defined by the ordinary projective space $
\mathbb{P}(V_0)$ and structure sheaf
\[ \mc{O}_{\mathbb{P}(V)} \cong \Sym((V_1)^\vee \otimes 
\mathcal{O}_{\mathbb{P}(V_0)}(-1))\;. \]
If $V = k^{m+1|n}$ with coordinates $x_0, \dots, x_m \, | \, 
\xi_1, \dots, \xi_n$, then 
\[ \mathbb{P}^{m|n} = \Proj{k[x_0, \cdots, x_m \, | \,  \xi_1, 
\cdots, \xi_n]}, \]
where all the generators have degree one in the $\mathbb{Z}$-
grading.

\subsection{Functor point of view}

Let $X$ be a superscheme.
We will denote by $h_X$ the \textit{representable functor} i.e. functor naturally isomorphic to the functor 
$$ 
\Hom(\underline{\hspace{0.4cm}}\,, X): \SSch^\text{op} \to \mathbf{Set}; S \mapsto X(S) := \Hom(S, X).
$$ 
If $S := \Spec{A}$ then the $A$-valued points of $X$ is defined as $X(A) := h_X(\Spec{A})$. 
Since superscheme has open cover by affine superschemes so it is enough to study the restriction of this functor to $\salg$ via the functor $\Spec{\underline{\hspace{0.4cm}}}$ i.e.
$$h_X(\underline{\hspace{0.4cm}}) := \Hom(\Spec{\underline{\hspace{0.4cm}}}, X): \salg^ \text{op} \to \mathbf{Set}; A \mapsto X(A)\,.$$
Furthermore, this restricted functor is representable if $X$ is an affine superscheme. 

Let $\iota \colon \Sch \ra \SSch$ be the canonical functor mapping an ordinary scheme to a superscheme with zero odd structure sheaf. From the definition of the bosonic reduction and the inclusion canonical functor $\iota$, we have

\begin{proposition}\label{prop-bosonic-reduction}
Let $X$ be a superscheme. Then, there exists a natural isomorphism of functors from $\SSch^\text{op}$ to $\Set$:
$$
\Hom_{\SSch}(\iota(\underline{\hspace{0.4cm}}),\, X) \cong
\Hom_{\Sch}(\underline{\hspace{0.4cm}},\, X_\mathrm{bos})
$$
\end{proposition}

\subsubsection{Affine super spaces: }
Let $A \in \salg$ and let $V=V_0 \oplus V_1$ be a
free supermodule (over $k$). Let $\smod$ denote
the category of $k$-super vector spaces.

Define $F_V: \salg \to \mathbf{Set}; A \mapsto F_V(A)$ where 
$$
F_V(A) = (A \otimes V)_0 = A_0 \otimes V_0 \oplus A_1 \otimes V_1.
$$
In general this functor is not representable. However, if $V$ is finite
dimensional we have:
$$
(A \otimes V)_0 \cong 
\Hom_{\smod}(V^*,A) \cong \Hom_{\salg}(\Sym(V^*),A) \cong 
h_{\Spec{\Sym(V^*)}}(A) = h_{\mathbb V}(A)
$$
where $\Sym(V^*)$ denotes the symmetric algebra over the dual space $V^*$.
In particular, $\Spec{\Sym(V^*)}(k) = V_0$.

In particular, if we fix a basis for a finite dimensional super vector space $V$ of dimension $m|n$.
The functor $F_V$ is represented by 
$k[x_1 \dots x_m, \xi_1 \dots \xi_n]$
where $x_i$ and $\xi_j$ are respectively even and odd indeterminates.

\subsubsection{Supermatrices: }
Let $A \in \salg$.
Define $\rM_{m|n}(A)$ as the set of endomorphisms of the $A$-supermodule
$A^{m|n}$. Choosing coordinates we can write:
$$
\rM_{m|n}(A)=
\left\{ \begin{pmatrix} a & \alpha \\ \beta & b \end{pmatrix} 
\right\}
$$
where $a$ and $b$ are $m \times m$, $n \times n$ blocks of even elements
and $\alpha$, $\beta$ are $m\times n$, $n \times m$ blocks of odd elements.

This is the functor of points
of an affine supervariety represented by the commutative
superalgebra: $k[\rM(m|n)]=$ $k[x_{ij},\xi_{kl}]$
where $x_{ij}$'s
and $\xi_{kl}$'s are respectively even and odd variables
with
$1 \leq i,j \leq m$ or  $m+1 \leq i,j \leq m+n$,
$1 \leq k \leq m$, $m+1 \leq l \leq m+n$ or
$m+1 \leq k \leq m+n$, $1 \leq l \leq m$.

Notice that $\rM_{m|n} \cong h_{\bA^{m^2+n^2|2mn}}$.

\subsection{Representations and super group action}
Recall that any $g\in \GL[(A)]$ is of the form
\[
\begin{pmatrix}
a_{00} & a_{01} \\ a_{10} & a_{11} 
\end{pmatrix}
\]
where $a_{00}$ and $a_{11}$ are $p \times p$, $q \times q$ 
blocks of even elements and $a_{01}$, $a_{10}$ are $p\times q$,
$q \times p$ blocks of odd elements.
Consider the group functor 
$$
\GL: \salg \ra \grps, \quad A\mapsto \GL[(A)] \,;
$$

Let $V$ be a super vector space over $k$. The super vector space can be viewed as the functor that assigns to any commutative superalgebra $A$ the Grassmann envelop $V(A) = (V\otimes A)_0$. 
We define the group functor $\GLV : \salg \ra \grps$ that assigns to each commutative superalgebra $A$, the even invertible elements of $\mathrm{End}_A(V \otimes A)$. If $V$ is finite-dimensional 
and if we fix a basis of $V$ , we have an isomorphism 
$V \simeq k^{p|q}$ for some $p$ and $q$, which induces an isomorphism of group functors $\GLV \simeq \GL$. Therefore, 
$\GLV$ is an affine algebraic supergroup for finite-dimensional 
$V$. The action of $\GLV[(A)]$ on $(V \otimes A)_0$ is same as 
a natural transformation $\GLV \times V \ra V $.

Let $G$ be a group functor and $V$ a super vector space, viewed as a functor. A linear representation of $G$ in $V$ is a morphism of group functors $G \ra \GLV $.

\begin{theorem}\cite{We09}
Let $G$ be an affine supergroup scheme with representing super Hopf algebra $k[G]$. Then a linear representation 
$\Phi\colon G \ra \GLV$ on $V$ corresponds to a unique $k$-linear map 
$\rho\colon V \ra V \otimes_k k[G]$ of super vector spaces such that the following diagrams commute:
\[
\xymatrix@C=3.0pc @R=3.0pc{
V \ar[r]^\rho \ar[d]_\rho & V\otimes_k k[G] \ar[d]^{1\otimes \Delta} \\
V\otimes_k k[G] \ar[r]_-{\rho\otimes 1} & V\otimes_k k[G]\otimes_k k[G]
}
~~~~~~~~~
\xymatrix@C=3.5pc{
V \ar[r]^\rho \ar[rd]_\simeq &
V\otimes_k k[G] \ar[d]^{1\otimes \epsilon} \\
& V\otimes_k k
}
\]
\end{theorem}

Let $\Gm = \mathrm{GL}_{1|0}$. Given a super vector space $V$ and
an integer $m\in \Z$, consider the map
$$
V \ra V\otimes_k k[t, t^{-1}]; \quad \quad v\mapsto v\otimes t^m \;.
$$
This defines a representation of $\Gm$, called a representation of \emph{weight} $m$. 

\begin{proposition}\label{prop-Gm}
    Every super representation $V$ of $\Gm$ is a direct sum $V = \ \bigoplus_{m\in \Z} V_{(m)}$, where each $V_{(m)}\subseteq V$
    is a sub-representation of weight $m$.
\end{proposition}
\begin{proof}
Since graded $k$-linear morphism $V \to V \otimes_k k[t, t^{-1}]$ restricts to $V_0$
and $V_1$, and hence the assertion follows from the classical even case.
\end{proof}
\subsection{Invariant superspace}
Let $\G$ be a 
group superscheme, and $\X$ be a superscheme. We say that $\G$ acts on $\X$ if there is a functor morphism 
$\mu \colon \G \times \X \ra \X$ with the following properties: 
For $R\in \mathbf{SRings},\; x\in \X(R), \; g,h\in \G(R)$, we have
\begin{enumerate}
\item $\mu_R(e, x) = x$
\item $\mu_R(g, \mu_R(h, x)) = \mu_R(gh, x)$.
\end{enumerate}

Suppose an affine supergroup scheme $G$ is acting on an affine super scheme $X = \Spec A$, where $A$ is a super $k$-algebra.
Therefore, there is a morphism of super $k$-algebras 
$\rho\colon A\ra A\otimes_k k[G]$ such that the following condition 
holds:
\begin{enumerate}
    \item The composition $A \stackrel{\rho}{\ra} A\otimes k[G] \stackrel{1\otimes \varepsilon}{\ra} A\otimes_k k \stackrel{\simeq}{\ra} A$ is equal to the identity.
    \item The following diagram commutes:
    \[
    \xymatrix@C=3.0pc @R=3.0pc{
A \ar[r]^\rho \ar[d]_\rho & A\otimes_k k[G] \ar[d]^{\rho\otimes 1} \\
A\otimes_k k[G] \ar[r]_-{1\otimes \mu} & A\otimes_k k[G]\otimes_k k[G]
}
    \]
    where $\mu$ is co-multiplication on $k[G]$.
\end{enumerate}

The subset of $G$-invariants 
$$
A^G := \{f\in A\;|\; \rho(f) = f\otimes 1\}
$$
is a super subring of $A$. The affine scheme $\Spec{A^G}$ is called the invariant superspace.

\section{Super GIT}

\subsection{Affine super GIT}
Let $G$ be a linear reductive group scheme acting on an affine super scheme $X = \Spec A$, where $A$ is a super $k$-algebra. Then the affine super geometric invariant quotient (GIT) is defined as the affine super scheme $X / \! \!/ G:= \Spec{A^G}$. It comes with a canonical map $\pi: \Spec{A} \to \Spec{A^G}$ induced by the canonical inclusion of super algebras $i: A^G \hookrightarrow A$.

\begin{proposition}\label{prop:affine_GIT}
    The canonical map $\pi: X \to X / \! \!/ G$ is a $G$-invariant morphism. If $\phi:X \to Z$ is a $G$-invariant morphism from $X$ to any super scheme $Z$, then there exists a unique morphism $\psi: X / \! \!/ G \to Z$ such that $\phi = \psi \circ \pi$.
\end{proposition}
\begin{proof}
    
Let $\mu: G \times X \to X$ denote the super-action morphism and $p_2: G \times X \to X$ denote the projection onto the second factor. The condition that $\pi$ is $G$-invariant translates to the identity $\pi \circ \mu = \pi \circ p_2$. 

Passing to the coordinate rings, we must verify that the following two superalgebra homomorphisms from $A^G$ to $A \otimes k[G]$ coincide:
\[ \hat{\mu} \circ i \quad \text{and} \quad \hat{p_2} \circ i \]
By definition, $\hat{\mu}$ is precisely the right super-coaction map $\rho$. For any homogeneous element $f \in A^G$, we have $\hat{\mu}(i(f)) = \rho(f) = f \otimes 1$. On the other hand, the pullback of the projection map acts as $\hat{p_2}(i(f)) = f \otimes 1$.

Since $\hat{\mu}(i(f)) = \hat{p_2}(i(f))$ for all elements in $A^G$, the superalgebra homomorphisms are identical. Thus, $\pi \circ \mu = \pi \circ p_2$, establishing $G$-invariance.

Next, we want to prove the universal property. Assume first that $Z$ is an affine superscheme, so $Z = \Spec{B}$
for some supercommutative algebra $B = B_0 \oplus B_1$. Let $\psi: X \to Z$ be a $G$-invariant morphism, which corresponds to an even superalgebra homomorphism $\hat{\psi}: B \to A$. 

The $G$-invariance of $\phi$ implies that $\hat{\mu} \circ \hat{\phi} = \hat{p_2} \circ \hat{\phi}$. Evaluating this on an arbitrary homogeneous element $b \in B$, we obtain:
\[ \rho(\hat{\phi}(b)) = \hat{\phi}(b) \otimes 1 \]
Therefore, $\hat{\phi}
(B) \subseteq A^{G}$. 

Consequently, $\hat{\phi}$ factors uniquely through the natural inclusion $i: A^G \hookrightarrow A$. This yields a unique even superalgebra homomorphism $\hat{\psi}: B \to A^{G}$ making the following diagram commute:
\[
\begin{tikzcd}
X \arrow[rr, "\phi"] \arrow[d, "\pi"'] & & Z
\\
X / \! \!/ G \arrow[rru, "\exists! \, \psi"'] & & 
\end{tikzcd}
\quad \quad \quad
\begin{tikzcd}
A & & B \arrow[ll, "\hat{\phi}"'] \arrow[lld, "\exists! \, \hat{\psi}"] \\
A^G \arrow[u, "i"] & & 
\end{tikzcd}
\]
Setting $\psi = \Spec{\hat{\psi}}$ satisfies the claim for the affine target case.

Now let $Z$ be a general, non-affine superscheme. We cover $Z$ by a family of open affine sub-superschemes $Z = \bigcup W_i$, where $W_i = \Spec{B_i}$. 

Since $\phi: X \to Z$ is $G$-invariant morphism, the preimages $U_i = \phi^{-1}(W_i)$ define $G$-stable open sub-superschemes of $X$.

Because $G$ is a reductive group, the sub-superalgebra of invariants $A^{G}$ is finitely generated (generalising the classical Hilbert-Nagata theorem via the exactness of the Reynolds operator on superrepresentations). The morphism $\pi$ maps $G$-stable open affine sub-superschemes $U_i$ directly to an open affine cover $\pi(U_i)$ of $X / \! \! / G$. 

Restricting to each component, the affine target case yields a unique morphism of superschemes $\psi_i: \pi(U_i) \to W_i \subseteq Z$ matching $\phi|_{U_i}$. By the uniqueness of the factorization on intersections $\pi(U_i) \cap \pi(U_j)$, these local morphisms glue together to give morphism of superschemes $\psi: X / \! \! / G \to Z$ satisfying $\phi = \psi \circ \pi$.
\end{proof}
\subsection{Stability and Semi-stability}
Let $G$ be a linearly reductive group acting on a super vector space $V = V_0 \oplus V_1$ with the kernel $\Delta$. Then, it induces an
action on the corresponding affine super scheme 
$\mathbb V = \Spec{\Sym(V^*)}$. Let $A:= \Sym(V^*)$ and $A_0:= \Sym(V_0^*)$. 

Note that a point $x\in \mathbb{V}(k)$ corresponds to a map $ev_x : A\ra k$, which factor through $A_0$.

Fix a character $\chi: G \to \mathbb{G}_m$ of a linearly reductive group $G$. Let $g_\chi \in k[G]$ be the group-like element representing the character $\chi$.

\begin{definition}\label{def:stab}
\begin{enumerate}
    \item[(i)] A point $x \in \mathbb{V}(k)$ is $\chi$-semistable if there is a relative invariant (or $\chi$-semi-invariant of weight $n$) $$f \in A^{G,\chi^{n}} := \{ f \in A | \rho(f) = f \otimes g_\chi^{n} \}$$  with $n \ge 1$, such that $f(x) \ne 0$.
Equivalently, a point $ev_x: A \to k$ (which factors through $A_0$ and hence a map $x: A_0 \to k$) is semistable if there is $f \in A^{G,\chi^{n}}$ for some $n \geq 1$, such that $ev_x(f) \ne 0$.

    \item[(ii)] A point $x \in \mathbb{V}(k)$ is $\chi$-stable if there is a relative invariant $f \in A^{G,\chi^{n}}$ with $n \ge 1$, such that $ev_x(f) \ne 0$ and, further, $\dim G \cdot x = \dim G/\Delta$ and the $G$-action on $\{x \in \mathbb{V}(k) \mid ev_x(f) \ne 0\}$ is closed.
\end{enumerate}
\end{definition} 

Consider $\tilde{A} := A[t]$ with the extended coaction $\tilde{\rho} : \tilde{A} \to \tilde{A} \otimes H$ defined by $\tilde{\rho}(a) = \rho(a)$ and $\tilde{\rho}(t) = t \otimes g_{\chi}^{-1}$. A lift of a point $x\in \mathbb{V}(k)$ is a point $\hat{x} = (x,z) \in \mathbb{V}(k) \times k$ corresponding to the algebra map $ev_{\hat{x}} : \tilde{A} \to k$, where $ev_{\hat{x}}(a) = ev_x(a)$ and $ev_{\hat{x}}(t) = z \neq 0$.

\begin{lemma}\label{lemma-ss-criterion}
Lift $x \in \mathbb{V}(k)$ to a point $\hat{x} = (x,z) \in \mathbb{V}(k) \times k$ with $z \ne 0$. Then
\begin{enumerate}
    \item[(i)] $x$ is $\chi$-semistable if and only if the orbit closure $\overline{G \cdot \hat{x}} \subseteq \mathbb{V}(k) \times k$ is disjoint from the zero-section $\mathbb{V}(k) \times \{0\}$. In particular, it is necessary that $\chi(\Delta)=\{1\}$.
    \item[(ii)] $x$ is $\chi$-stable if and only if $G \cdot \hat{x}$ is closed and the stabiliser of $\hat{x}$ contains $\Delta$ with finite index.
\end{enumerate}
\end{lemma}
\begin{proof}
(i) If $x$ is $\chi$-semistable, there exists $f \in A$ and $n \geq 1$ such that $\rho(f) = f \otimes g_{\chi}^n$ and $ev_x(f) \neq 0$. Define $F = f \cdot t^n \in \tilde{A}$. Then,
\[
\tilde{\rho}(F) = \tilde{\rho}(f) \tilde{\rho}(t^n) = (f \otimes g_{\chi}^n)(t^n \otimes g_{\chi}^{-n}) = f t^n \otimes (g_{\chi}^n g_{\chi}^{-n}) = F \otimes 1
\]
Thus, $F$ is an invariant. Since $ev_{\hat{x}}(F) = f(x)z^n \neq 0$, the point $\hat{x}$ is not in the zero set of $F$. However, for any point $(v, 0)$ in the zero-section, $ev_{(v, 0)}(F) = f(v) \cdot 0^n = 0$. Since invariants are constant on orbit closures, the closure of the orbit of $\hat{x}$ cannot intersect the zero-section.

Conversely, if the orbit closure is disjoint from the zero-section, by the linearly reductive property, there exists an invariant $F = \sum f_i t^i$ such that $ev_{\hat{x}}(F) \neq 0$ and $F$ vanishes on the zero-section (the ideal $(t)$). The condition $F \in (t)$ implies that $f_0 = 0$. Since $\tilde{\rho}(F) = F \otimes 1$, by comparing powers of $t$, we have $\rho(f_i) = f_i \otimes g_{\chi}^i$. Since $ev_{\hat{x}}(F) \neq 0$, it follows that $ev_{\hat{x}}(f_n t^n) \neq 0$ for some $n\ge 1$. This proves that $x$ is $\chi$-semistable.

(ii) 
If we take the induced action on the bosonic reduction, the above topological conditions are equivalent to stability by considering the case of usual schemes $x \in \mbox{Spec}(A_0)(k)$ with the action of an affine algebraic group scheme $G$; see \cite[Lemma 2.2]{Ki94}.
\end{proof}

Fix a character $ \chi: G \to \mathbb{G}_m $ of a linearly reductive group $G$. For a 1-PS $ \lambda : \mathbb{G}_m \to G $, we define $\langle \chi, \lambda \rangle$ to be the integer such that:
$$ \chi(\lambda(t)) = t^{\langle \chi, \lambda \rangle} $$

\begin{proposition}\label{prop-pairing}
A point $x \in \mathbb{V}(k)$ is $\chi$-semistable if and only if $\chi(\Delta)=\{1\}$ and every one parameter subgroup $\lambda$ of $G$, for which $\lim_{t \rightarrow 0} \lambda(t) \cdot x$ exists, satisfies $\langle \chi, \lambda \rangle \ge 0$. Such a point is $\chi$-stable if and only if the only one-parameter subgroups of $G$, for which $\lim_{t \rightarrow 0} \lambda(t) \cdot x$ exists and $\langle \chi, \lambda \rangle = 0$, are in $\Delta$.
\end{proposition}
\begin{proof}
Since $\chi$ and $\lambda$ are an even morphisms, the result follows from Lemma \ref{lemma-ss-criterion} and \cite[Proposition 2.5]{Ki94}.
\end{proof}

The super GIT quotient of $\mathbb{V}$ by $G$ can be described as 
$$
\mathbb{V} / \! \!/_{\! \!\chi} G = \Proj{\bigoplus_{n\geq 0} A^{G, \chi^n}}\;.
$$
The natural morphism $\phi \colon \mathbb{V} / \!\! /_{\! \!\chi} G \ra \Spec{A^G}$ is super projective (\cite[Definition 2.23]{BRP}). 

\begin{proposition}\label{prop-git-co-rep}
There exists a $G$-invariant morphism $\pi: \mathbb{V}^{ss} \to \mathbb{V} / \! \!/_{\! \!\chi} G$ of super schemes such that the points of $\mathbb{V} / \!\! /_{\! \!\chi} G (k)$ are in bijective correspondence with the GIT equivalence classes of points in $\mathbb{V}^{ss}(k)$. 
The super scheme $\mathbb{V} / \!\! /_{\! \!\chi} G$ corepresents the quotient functor $h_{\mathbb{V}^{ss}}/ h_G: \mathsf{\SSch}^\text{op} \to \Set$ defined by $S \mapsto \mathbb{V}^{ss}(S) / G(S)$.
\end{proposition}
\begin{proof}
We need to show that for any super scheme $Z$, there is a natural bijection:
\[ 
\text{Hom}(\mathbb{V} / \!\! /_{\! \!\chi} G, \;Z) \cong \text{Hom}_G(\mathbb{V}^{ss}, \;Z) \,.
\]
Let $R = \bigoplus_{n\geq 0} A^{G, \chi^n}$. The super scheme $\mathbb{V} / \!\! /_{\! \!\chi} G$ is covered by affine super schemes $U_f = \Spec{(R_f)_0}$ for even invariants $f \in R_n$. The preimage of $U_f$ in $\mathbb{V}^{ss}$ is the affine open set $\mathbb{V}_f = \{x \in \mathbb{V}(k) \mid f(x) \neq 0\}$. The map $\pi: \mathbb{V}_f \rightarrow U_f$ corresponds to the ring inclusion:
\[ 
(R_f)_0 \hookrightarrow A_f 
\]
Let $\phi \colon \mathbb{V}^{ss} \rightarrow Z$ be a $G$-invariant morphism.
We want to construct a unique $\psi \colon \mathbb{V} / \!\! /_{\! \!\chi} G \to Z$ such that $\phi = \psi \circ \pi$.

If we consider the open cover by affine super schemes $\mathbb{V} / \!\! /_{\! \!\chi} G = \bigcup_i U_{f_i}$ for some homogeneous generators $f_i$ of $R$, this gives an open cover $\mathbb{V}^{ss} = \bigcup_i \mathbb{V}_{f_i}$. Hence we get ring homomorphisms $\hat{\pi_i}: (R_{f_i})_0 \hookrightarrow A_{f_{i}}$ corresponding to restriction of $\pi$. Note that $(R_{f_i})_0 = A_{f_i}^G$.
This means that the restriction $\pi: \mathbb{V}_{f_i} \to U_{f_i}$ is an affine super GIT. Hence, using the propositon \ref{prop:affine_GIT} we get unique morphisms $\psi_i: H_{f_i} \to Z$ such that $\phi_i = \psi_i \circ \pi$. Now, using the universal property of $\pi$ in Proposition \ref{prop:affine_GIT}, we can glue the morphisms $\psi_i$ to obtain the desired morphism $\psi$. 
\end{proof}
\section{Super-representations of quivers}
Let $k$ be an algebraically closed field of characteristic $p\neq 2$. By a quiver $Q$, we mean a finite
directed graph. More precisely, A quiver is a quadruple $Q = (Q_0, Q_1, s, t)$, where $Q_0$ is a finite 
set of vertices, $Q_1$ is a finite set of arrows, $s$ and $t$ are source and target maps from $Q_1$ 
to $Q_0$, respectively.

A super-representation $M$ of $Q$ consists of a family of 
super-vector spaces $M_v$ indexed by the vertices $v\in Q_0$ together with a family 
of linear maps 
\[
M_a := \begin{pmatrix}
    M_a^{00} & M_a^{10} \\
    & \\
    M_a^{01} & M_a^{11}
\end{pmatrix}: M_{s(a)} \to M_{t(a)}, \quad M_a^{ij}\in \Hom_k(M_{s(a)}^i, M_{t(a)}^j), \quad i,j= 0,1
\]
indexed by the arrows
$a \in Q_1$. If $M$ and $N$ are two super-representations of $Q$, then
a morphism $\phi \colon M \ra N$ is a collection of superspace morphisms, i.e., parity-preserving, 
$\phi_v \colon M_v \ra N_v$ for $v \in Q_0$ such that for every $a\in Q_1$, 
the following diagram
\[
\xymatrix{
M_{s(a)}\ar[r]^{\phi_{s(a)}} \ar[d]_{M_a} & N_{s(a)}\ar[d]^{N_a}\\
M_{t(a)} \ar[r]_{\phi_{t(a)}} & N_{t(a)}
}
\]
commutes.

Let $\mathsf{SRep}(Q)$ be the category of super- representations of the quiver $Q$. Let 
$$
\mathsf{Rep}(Q, \smod)
$$ 
be the category of representations of $Q$ in the category of super-vector spaces over $k$. 
Note that $\mathsf{Rep}(Q, \smod)$ is full subcategory of 
$\mathsf{SRep}(Q)$.

Let $\widehat{Q}$ denote an \emph{edge doubled} quiver with 
\[
\widehat{Q}_0=Q_0, \quad \widehat{Q}_1=\{a_0\mid a\in Q_1\}\sqcup \{a_1\mid a\in Q_1\}.
\] 
Then the \emph{path algebra} $k\widehat{Q}$ has the unique superalgebra structure, such that $|a_i|=i$, $a\in Q_1$, $i\in\{0, 1\}$, and for any $v\in Q_0$ the idempotent $a_v$ is even.

\begin{lemma} [Lemma 3.1, \cite{BZ20}]
The abelian category $\mathsf{SRep}(Q)$ of super-representations of the quiver $Q$ is equivalent to
the category of $k\widehat{Q}-\mathsf{Smod}$.
\end{lemma}

Let $\sdim M$ denote the {\it super-dimension vector} of $M$ defined by
$(\sdim M)(v)= \sdim M_v$ for $v\in Q_0$.

\begin{remark}\label{Rem-tilde}
The abelian category $\mathsf{SRep}(Q)$ of super-representations 
of a quiver $Q$ is equivalent to $\mathsf{Rep}(\tilde{Q})$, where 
$\tilde{Q}$ is the edge and vertex-doubled quiver with vertices 
$$
\tilde{Q}_0 = \{v_i \mid v \in Q_0, i=0,1\}
\quad \mbox{and} \quad   
\tilde{Q}_1 = \{a_{ij} \mid a \in Q_1, i,j=0,1\}, \mid a_{ij} \in \tilde{Q}_1
$$
with $s(a_{ij}) = s(a)_i$ and $t(a_{ij}) = t(a)_j$. 
The equivalence is defined by 
$$
M \mapsto \tilde{M}, \mid \tilde{M}_{v_i} = M_v^i,
$$ and mapping 
$$
\tilde{M}_{a_{ij}} = M_a^{ij}
$$ for $v\in Q_0, a\in Q_1$ and $i,j = 0,1$. 
\end{remark}

\subsection{Super-representation space}
Let $\sd \in (\N \times \N)^{Q_0}$ be a fixed super-dimension vector. 
That is for each vertex $v\in Q_0$, we have $\sd(v) = (\d^0_v | \d^1_v)$.

Let 
$$
\bS \bfR(\sd):= \displaystyle \bigoplus_{a\in Q_1} 
\Hom_k(k^{\sd(s(a))}, (k^{\sd(t(a))})
$$
be a super-vector space of super-representations of $Q$ having super-dimension vector $\sd$. Notice that the elements of this super-vector space give a bijection with super-representations of $Q$ having super-dimension vector $\sd$.

The supergroup 
$$
\mathrm{GL}(\sd) := \prod_{v\in Q_0} \mathrm{GL}(\sd(v))
$$
acts naturally on the super-vector space $\bS \bfR(\sd)$.

The isomorphism classes of super representations of $Q$ having dimension vector $\sd$ correspond to the orbits of $\bS \bfR(\sd)$
with respect to the action of reduced subgroup $\mathrm{GL}(\sd)_{\mathrm{ev}}$ of $\mathrm{GL}(\sd)$. Note that the reduced subgroup $\mathrm{GL}(\sd)_{\mathrm{ev}}$ is isomorphic to 
$G(\sd) := \prod_{v\in Q_0} \mathrm{GL}(\d^0_v)\times \mathrm{GL}(\d^1_v)$.

Let 
$$
\bbSR(\sd) = \bS \bfR(\sd) \bigoplus \Pi (\bS \bfR(\sd))
$$
be the super-vector space, where $\Pi$ denote the parity shift operator. The group $G(\sd)$ acts on $\bS \bfR(\sd)$ as follows
\[
(g_v^{00}, g_v^{11}) \cdot M := \begin{pmatrix}
    g_{t(a)}^{00} M_a^{00} (g_{s(a)}^{00})^{-1} & g_{t(a)}^{00} M_a^{10} (g_{s(a)}^{11})^{-1} \\
    & \\
    g_{t(a)}^{11} M_a^{01} (g_{s(a)}^{00})^{-1} & g_{t(a)}^{11} M_a^{11} (g_{s(a)}^{11})^{-1} 
\end{pmatrix} 
\]
This, in turn, gives the diagonal action on $\bbSR(\sd)$
\[ 
(g_v^{00}, g_v^{11}) \cdot (M, \Pi M) := 
((g_v^{00}, g_v^{11}) \cdot M, \Pi ((g_v^{00}, g_v^{11}) \cdot M))
\]

Let $\bfSR(\sd)$ be the affine super-scheme corresponding to the 
super-vector space $\bbSR(\sd)$. It is easy to see that the underlying
topological space of $\bfSR(\sd)$ is in bijection with the super-vector 
space $\bS \bfR(\sd)$.

\subsection{Semistable super-representations}
Fix $\btheta := (\btheta(v))_{v \in Q_0} = (\btheta_v^0, \btheta_v^1)_{v \in Q_0} \in (\Z\times\Z)^{Q_0}$. For a super-representation $M$ of 
a quiver $Q$, we define
$$
\btheta(M) := \sum_{v\in Q_0} (\btheta_v^0 \dim M_v^0 + \btheta_v^1 \dim M_v^1)\,.
$$

\begin{definition}\label{def-theta-ss}\rm{
We say that a super-representation $M$ of a quiver $Q$ is 
$\btheta$-\emph{semistable} if $\btheta(M) = 0$ and every sub 
super-representation $M'\subseteq M$ satisfies $\btheta(M') \geq 0$. 
We say that $M$ is $\btheta$-\emph{stable} if the only sub super-representations $M'$ with $\btheta(M') = 0$ are $M$ and $0$.
}
\end{definition}

\begin{proposition}\label{prop-btheta-ss-AC}
The $\btheta$-semistable representations of fixed super-dimension $\sd$ form an abelian subcategory of 
$\mathsf{SRep}(Q)$. Moreover, the simple objects in this subcategory 
are precisely the $\btheta$-stable representations.
\end{proposition}
\begin{proof}
Using Remark \ref{Rem-tilde}, this follows from the results of King \cite{Ki94}. 
More precisely, under the equivalence in Remark \ref{Rem-tilde}, the super-representations of the quiver $Q$ with super-dimension vector $\sd$ corresponds to the representations of the quiver $\tilde{Q}$ with dimension vector $\d = (\d^{0}_{v_{0}^{1}}, \d^{1}_{v_{1}^{1}}, \dots , \d^{0}_{v_{0}^{n}}, \d^{1}_{v_{1}^{n}})$, where $Q_0 = \{v^1, \dots, v^n\}$. The stability parameter $\btheta$ yields a stability parameter $\theta = (\btheta^{0}_{v_{0}^{1}}, \btheta^{1}_{v_{1}^{1}}, \dots , \btheta^{0}_{v_{0}^{n}}, \btheta^{1}_{v_{1}^{n}})$ for the quiver $\tilde{Q}$.

From the definition, it follows that under the equivalence $M\mapsto \tilde{M}$, the super representation $M$ is $\btheta$-(semi)stable if and only if the representation $\tilde{M}$ of $\tilde{Q}$ is $\theta$-(semi)stable. 
\end{proof}
The category of $\btheta$-semistable representations is Noetherian and
Artinian, and hence the Jordan-H\"older theorem holds, and so we have 
a notion of $S$-equivalence in this category.

\section{Moduli of super representations}
Let $S$ be a super-scheme with structure sheaf $\mathcal{O}_S = \mathcal{O}_{S,\text{0}} \oplus \mathcal{O}_{S,\text{1}}$. 
Recall that a sheaf $\mathcal{F}$ of super $\mathcal{O}_S$-modules is called locally free (or a super vector bundle) of rank $p|q$ if every point of $S$ has an open neighbourhood $U$ such that 
$$\mathcal{E}\vert _{U} \cong \mathcal{O}_S ^p \oplus (\Pi \mathcal{O}_S)^q.$$

Let $\mathcal{E}$ be a locally free $\mathcal{O}_{S}$-module of rank $p\vert{}q$. Then, we have (locally on split chart)
\[
\mathcal{E} \cong \left(\mathcal{E}_{\mathbf{0}} \otimes_{\mathcal{O}_{S,\text{0}}} \mathcal{O}_{S}\right) \oplus \left(\mathcal{E}_{\mathbf{1}} \otimes_{\mathcal{O}_{S,\text{0}}} \mathcal{O}_{S}\right)
\]
where $\mathcal{E}_{\mathbf{0}}$ and $\mathcal{E}_{\mathbf{1}}$ are $\mathcal{O}_{S,\text{0}}$-submodules.

A \textit{strict diagonal morphism} between locally free $\mathcal{O}_{S}$-modules $\mathcal{E}$ and $\mathcal{F}$ is an $\mathcal{O}_{S}$-linear parity-preserving sheaf morphism $\Phi: \mathcal{E} \to \mathcal{F}$ that locally strictly preserves the splitting sub-sheaves under the action of $\mathcal{O}_{S,\text{0}}$ :
\[
\Phi (\mathcal{E}_{\mathbf{0}}) \subseteq \mathcal{F}_{\mathbf{0}} \quad \text{and} \quad \Phi (\mathcal{E}_{\mathbf{1}}) \subseteq \mathcal{F}_{\mathbf{1}}
\]

A flat family of super-representations over a super scheme $S$ is a 
locally-free $\stsheaf{S}$-module $\sheaf{F}$ together with a super $k$-algebra
homomorphism $k\widehat{Q}\ra \mathrm{End}(\sheaf{F})$.

We say that two families $\mathcal{F}_1$ and $\mathcal{F}_2$ of 
super-representations over a super scheme $S$ are strictly isomorphic 
if there exists strictly diagonal isomorphism
$\Phi\colon \mathcal{F}_1\ra \mathcal{F}_2$.

Let $\smf \colon \SSch^\mathrm{op}\ra \Set$ be the functor given by 
assigning to each super $k$-scheme $S$ the set of \emph{strict isomorphism} classes of 
flat families over $S$ of super-representations of $Q$ having dimension 
vector $\sd$.

Consider the following subfunctor 
$$
\smfss \colon \SSch^\text{op}\ra \Set ~(\mbox{resp. } 
\smfs \colon \SSch^\text{op}\ra \Set)
$$ 
of $\smf(\sd)$ defined by assigning to each super $k$-scheme $S$ the set 
of \emph{strict isomorphism} classes of flat families over $S$ of $\btheta$-semistable 
(resp. $\btheta$-stable) super-representations of $Q$ having dimension 
vector $\sd$.

Let $\mathbb{M}$ be the tautological family over $\bfSR(\sd)$. 
More precisely, the family $\mathbb{M}$ can be described as follows: 

For each vertex $v\in Q_0$, there is a trivial super vector bundle
$\mathbb{M}_v$ on $\bfSR(\sd)$ of rank $d_v^0|d_v^1$.
At any point $x\in \bfSR(\sd)(k)$, the fibre of $\mathbb{M}_v$
is the super vector space $k^{d_v^0|d_v^1}$.
For each arrow $a\in Q_1$, there is a morphism (need not be graded)
$\phi_a \, \colon \; \mathbb{M}_{s(a)} \ra \mathbb{M}_{t(a)}$
such that $\phi_a(x)$ is exactly the linear map $M_x(a)$
associated with the arrow $a$ in the super-representation $M_x$.

There is a natural functor $h\colon \fp{\bfSR(\sd)} \ra \smf$ 
(defined by $(f\colon S\ra \bfSR(\sd)) \mapsto [f^*\mathbb{M}]$, 
where $\mathbb{M}$ is a tautological family on $\bfSR(\sd)$).

\begin{proposition}\label{pro-local-isom}
The natural functor $h$ induces a local isomorphism $\tilde{h}\colon \fp{\bfSR(\sd)}/\fp{G(\sd)} \ra \smf$.
\end{proposition}
\begin{proof}
From the definition of action, it follows that two elements 
$\sigma, \tau \in \fp{\bfSR(\sd)}(S)$ gives 
$\sigma^*\mathbb{M}\cong \tau^*\mathbb{M}$ if and only if they 
are related by an elemnet of $\fp{G(\sd)}(S)$. Moreover, for any 
family of super-representations over a super scheme $S$,
there exists an open covering $S = \cup_i S_i$ such that the 
restriction to each $S_i$ is the pull-back by a map $S_i \ra \bfSR(\sd)$.
\end{proof}

Now, we translate the stability parameter $\btheta$ into the GIT context.

The stability condition $\btheta \in (\mathbb{Z} \times \mathbb{Z})^{Q_0}$ determines a character $\chi_\theta$ of the group $G(\sd) = \prod_{v \in Q_0} (\text{GL}(d_v^0) \times \text{GL}(d_v^1))$ defined by:
\[ \chi_{\btheta}(g) = \prod_{v \in Q_0} (\det g_v^0)^{\theta_v^0} \cdot (\det g_v^1)^{\theta_v^1} \]

\begin{proposition}\label{prop-1-ps-filtration}
Let $x$ be a point in $\bfSR(\sd)(k)$ corresponding to a super-representation $M \in \mathsf{SRep}(Q)$. Then, there is a surjection from the set
$$
\{1\text{-}PS ~\lambda\colon \mathbb{G}_m \ra G(\sd) ~\mbox{such that} ~ \lim_{t \to 0} \lambda(t) \cdot x ~\mbox{exist in} ~ \bfSR(\sd)\}
$$
to
$$
\{\mathbb{Z}\mbox{-filtrations of} ~ M ~\mbox{in the category}~ \mathsf{SRep}(Q)\}
$$
\end{proposition}
\begin{proof}
Let $\lambda \colon \mathbb{G}_m \ra G(\sd)$ be one-parameter subgroup such that $\lim_{t \to 0} \lambda(t) \cdot x = y$ in 
$\bfSR(\sd)$. By Proposition \ref{prop-Gm}, we get a 
$\mathbb{Z}$-grading 
    \[
    M_v = \bigoplus_{n \in \mathbb{Z}} M_v^{(n)} \quad \text{where} \quad M_v^{(n)} = (M_0)_v^{(n)} \oplus (M_1)_v^{(n)}
    \]
for each vertex $v\in Q_0$, where $\lambda(t)$ acts on the weight space $M_v^{(n)}$ as multiplication by $t^n$. 
    
For each arrow $a$, we have $M_a = \bigoplus M_a^{mn}$, where 
$
M_a^{(mn)}\colon M_{s(a)}^{(n)}\ra M_{t(a)}^{(m)}\,.
$
is a morphism of super vector spaces given by
\[
M_a^{(mn)} = \begin{pmatrix}
    M_{00}^{(mn)} & M_{10}^{(mn)} \\
    & \\
    M_{01}^{(mn)} & M_{11}^{(mn)}
\end{pmatrix} \in \Hom_k(M_{s(a)}^{(n)}, M_{t(a)}^{(m)})
\]
Therefore, by definition of the action, we have 
$$
\lambda(t)\cdot x = \bigoplus t^{m-n}M_a^{(mn)}\,.
$$
Since the limit $\lim_{t\rightarrow 0} \lambda(t)\cdot x$ exist, we have 
$M_a^{(mn)} = 0$ for all $m < n$. Let 
$$
M_v^{\geq n} := \bigoplus_{m\geq n} M_v^{(m)}\,.
$$
Then, $M_a$ gives a map $M_{s(a)}^{\geq n}\ra M_{t(a)}^{\geq n}$ for all $n$. 
These subspaces determine super sub-representations $M_n$ of $M$ for all $n$. 
This gives the $\mathbb{Z}$-filtration
$$
\cdots \supseteq M_n \supseteq M_{n+1} \supseteq \cdots
$$
where $M_n = M$ for $n \ll 0$ and $M_n = 0$ for $n \gg 0$ of $M$ in the category of super representations of $Q$. The limit $y$ is given by the graded super representation associated to the filtration
$$
M_y \cong \bigoplus_{n\in \Z} M_n/M_{n+1}\,. 
$$
Also, note that the filtration determined by $\lambda$ is proper, unless $\lambda$ is in $\Delta$.
\end{proof}

Suppose that 
$$
\btheta(\sd) = \sum_{v\in Q_0} (\btheta_v^0 \d_v^0 + \btheta_v^1 \d_v^1) = 0\;.
$$

\begin{proposition}\cite[cf. Proposition 3.1]{Ki94}\label{prop-king3.1}
A point in $\bfSR(\sd)(k)$ corresponding to a super-representation $M \in \mathsf{SRep}(Q)$ is $\chi_{\btheta}$-semistable (resp. $\chi_{\btheta}$-stable) if and only if $M$ is $\btheta$-semistable (resp. $\btheta$-stable).
\end{proposition}
\begin{proof}
Let $\lambda$ be a one-parameter subgroup of $G(\sd)$.
The pairing $\langle \chi_{\theta}, \lambda \rangle$ can be  expressed in terms of the associated filtration $(M_{n})_{n \in \mathbb{Z}}$ (see Proposition \ref{prop-1-ps-filtration}):
\begin{align*}
\langle \chi_{\btheta}, \lambda \rangle &= \sum_{v \in Q_{0}} \btheta_{v} \sum_{n \in \mathbb{Z}} n (\sdim M_{v}^{(n)}) \\
&= \sum_{n \in \mathbb{Z}} n \btheta(M_{n}/M_{n+1}) \\
&= \sum_{n \in \mathbb{Z}} \btheta(M_{n}).
\end{align*}
If a point x (corresponding to the super representation $M$) 
is $\btheta$-(semi)stable, the numerical criterion for stability (Proposition \ref{lemma-ss-criterion}) and the identity above immediately imply that $x$ is $\chi_{\btheta}$-(semi)stable.

Conversely, assume that $x$ is $\chi_{\btheta}$-(semi)stable. Let $M' \subset M$ be a super subrepresentation. We define a 1-step filtration associated with $M'$ as follows:
\begin{itemize}
    \item $M_n = M$ for $n < i$
    \item $M_n = M'$ for $n = i$
    \item $M_n = 0$ for $n > i$
\end{itemize}
This filtration is proper if and only if $0 \subsetneq M' \subsetneq M$. For a one-parameter subgroup $\lambda$ corresponding to this specific filtration, we have:
\[
\langle \chi_{\theta}, \lambda \rangle = \theta(M').
\]
Since $x$ is $\chi_{\btheta}$-(semi)stable, the assertion follows from Proposition \ref{lemma-ss-criterion}.
\end{proof}

If $M$ is $\btheta$-semistable, then a $\mathbb{Z}$-filtration $\{M_n\}$ corresponds to a one-parameter subgroup $\lambda$ with $\langle \chi_{\btheta}, \lambda \rangle = 0$ if and only if each $M_n$ is $\btheta$-semistable. This follows from the identity 
$$
\langle \chi_{\btheta}, \lambda \rangle = \sum_{n \in \mathbb{Z}} \btheta(M_n);
$$
since $\btheta(M_n) \geq 0$ for all $n$ by the semistability of $M$, the sum vanishes if and only if each $\btheta(M_n) = 0$.
The following result now follows from Proposition \ref{prop-1-ps-filtration}, \ref{prop-king3.1}, and standard GIT arguments.

\begin{proposition}\cite[cf. Proposition 3.2]{Ki94}\label{prop-git-S-equiv}
Let $M$ be a $\btheta$-semistable super-representation. Then:
\begin{enumerate}
\item[(i)] The $G(\mathbf{\sd})$-orbit of $M$ is closed in the semistable locus $\mathbf{SR}(\sd)^{\chi_{\btheta}\text{-ss}}$ if and only if $M$ is $\btheta$-polystable.
\item[(ii)] Two $\btheta$-semistable representations define the same point in the GIT quotient if and only if they are $S$-equivalent.
\end{enumerate}
\end{proposition}

\begin{theorem}
There exist super schemes $\sms \subseteq \smss$ corepresenting the respective moduli functors $\smfs \subseteq \smfss$. Furthermore,
the points in $\smss(k)$ correspond bijectively to the $S$-equivalence classes of $\btheta$-semistable super-representations $Q$ of super dimension vector $\sd$.
\end{theorem}
\begin{proof}
Using Proposition \ref{prop-king3.1}, it follows that the local isomorphism $\tilde{h}$ (see Proposition \ref{pro-local-isom}) induces a local isomorphism $\tilde{h}\colon \fp{\mathbf{SR}(\sd)^{\chi_{\btheta}\text{-ss}}}/\fp{G(\sd)} \ra \smfss$.
Let $\smss$ be the GIT quotient of 
$\mathbf{SR}(\sd)$ by $G(\sd)$. By Proposition \ref{prop-git-co-rep}, the super scheme $\smss$ co-represent the functor 
$\fp{\mathbf{SR}(\sd)^{\chi_{\btheta}\text{-ss}}}/\fp{G(\sd)}$. Since $\smfss$ is locally isomorphic to 
$\fp{\mathbf{SR}(\sd)^{\chi_{\btheta}\text{-ss}}}/\fp{G(\sd)}$, it follows that $\smss$ co-represent the functor
$\smfss$. By Proposition \ref{prop-git-S-equiv}, the last assertion follows.
\end{proof}

For a fixed super dimension vector $\sd$ for a quiver $Q$, 
we get a corresposnding dimension vector $\d$ for an associated quiver $\tilde{Q}$ (see Remark \ref{Rem-tilde}). The stability parameter $\btheta$ for $Q$ yieds a stability parameter $\theta$ for $\tilde{Q}$ (cf. Proposition \ref{prop-btheta-ss-AC}).
Let $\mathsf{M}^{{\theta}\text{-ss}}_\d$ (resp. $\mathsf{M}^{{\theta}\text{-st}}_\d$) be the moduli of 
$\theta$-semistable (resp. $\theta$-stable) representations of $\tilde{Q}$ having dimension vector $\d$ \cite{Ki94}. 
Recall that a family of $k\tilde{Q}$-modules over a scheme $S$ is a locally-free sheaf $\sheaf{F}$
over $S$ together with a $k$-algebra homomorphism $k\tilde{Q}\ra \mathrm{End}(\sheaf{F})$. The scheme $\mathsf{M}^{{\theta}\text{-ss}}_\d$ co-represent the moduli functor
$$
\Ms^{\theta\text{-}\mathrm{ss}}_\d \colon \Sch^\text{op}\ra \Set
$$
where $\Ms^{\theta\text{-}\mathrm{ss}}_\d(S)$ is the set of all isomorphism classes of families over $S$
of $\theta$-semistable representations with dimension vector $\d$. Similarly, the scheme $\mathsf{M}^{{\theta}\text{-st}}_\d$ represent the corresponding moduli functor $\Ms^{\theta\text{-}\mathrm{st}}_\d$.

Now, we have the following:

\begin{cor}
Assume that the quiver $Q$ has no oriented cycle. Then, the reduced scheme $(\smss)_{\mathrm{red}}$ (resp. $(\sms)_{\mathrm{red}}$) is isomorphic to 
$\mathsf{M}^{{\theta}\text{-ss}}_\d$ (resp. $\mathsf{M}^{{\theta}\text{-st}}_\d$).
\end{cor}
\begin{proof}
Let $B$ be the super coordinate ring of the super affine scheme 
$\mathbf{SR}(\sd)$ corresponding to the super vector space  
$\bS \bfR(\sd)$. Then, 
$$
B \cong \operatorname{Sym}(E^*) \otimes_k \bigwedge(O^*)\,,
$$
where $E$ is the even part and $O$ is the odd part of $\bS \bfR(\sd)$. 
Since $G(\sd)$ is classical linearly reductive group, it follows 
that taking invariants commutes with taking the quotient by the nilradical ideal $J$. Hence, we can conclude that 
\[
\left(\bigoplus_{n \ge 0} B^{G_{\sd}, \chi^n_{\btheta}}\right)_{\text{red}} \cong ~
\bigoplus_{n \ge 0} (B/J)^{G_{\sd}, \chi^n_{\btheta}} \,.
\]
From this, we have 
$$
(\smss)_{\mathrm{red}} = \Proj{\left(\bigoplus_{n \ge 0} B^{G_{\sd}, \chi^n_{\btheta}}\right)_{\text{red}}} \cong
\mathrm{Proj} \left(\bigoplus_{n \ge 0} (B/J)^{G_{\sd}, \chi^n_{\btheta}}\right)\,.
$$
Let $A$ be the coordinate ring of the ordinary representation space of the doubled quiver $\tilde{Q}$. Then, by definition of the action of $G(\sd)$, we have
$$
A^{G_{\d}, \chi^n_{\theta}} \cong \left(\operatorname{Sym}(E^*)\right)^{G_{\sd}, \chi^n_{\btheta}} \cong (B/J)^{G_{\sd}, \chi^n_{\btheta}}\,.
$$
Since $\mathsf{M}^{{\theta}\text{-ss}}_\d = \mathrm{Proj} \left(\bigoplus_{n \ge 0} A^{G_{\d}, \chi^n_{\theta}}\right)$, the result follows.

\end{proof}

\begin{example}
Let $Q=A_2 = s \xrightarrow{a} t$ be a linear quiver with two vertices and one arrow. The quiver $\Tilde{Q}$ will be a 
bipartite Quiver $K_{2,2}$
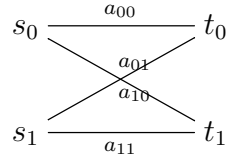
\[
\xymatrix{
s_0 \ar@{-}[rr]^{a_{00}} \ar@{-}[drr]^{a_{01}} && t_0 \\
s_1 \ar@{-}[urr]_
{a_{10}} \ar@{-}[rr]_{a_{11}} && t_1
}
\]
\end{example}
The representation space for $Q$ and $\tilde{Q}$ and their semiinvariants can be compared explicitly in this case.
If we denote by \[
k[x_{\cdot\cdot},y_{\cdot \cdot}, z_{\cdot \cdot}, w_{\cdot \cdot},\alpha_{\cdot \cdot}, \beta_{\cdot \cdot}, \gamma_{\cdot \cdot}, \delta_{\cdot \cdot}]
\]
the co-ordinate ring of the  super affine space $\mathbf{SR}(\sd)$, then we can identify the co-ordinate ring of its bosonic reduction with the polynomial ring 
\[
k[x_{\cdot \cdot},y_{\cdot \cdot}, z_{\cdot \cdot}, w_{\cdot \cdot}]
\]
and the induced conjugation actions of the group $G(\sd)$ on both spaces are the same. Hence we can get the identification of the $\chi_{\btheta}$-semi-invariant ring of $\tilde{Q}$ with the subring of $\chi_{\btheta}$-semi-invariant super ring of $Q$ generated by even variables. 
We observe the following:
\begin{enumerate}
    \item If we take $\sd = (1|1,1|1)$ and $\btheta = (-1, -1, 1, 1)$ then the nilpotent even function $f = \alpha \beta$ is $\chi_{\btheta}$-semi-invariant of weight one. 
    \item If we take $\sd = (1|2, 1|1)$ and $\btheta = (-2, 0, 1, 1)$, then the nilpotent odd function $f = x \alpha$ is $\chi_{\btheta}$-semi-invariant of weight one.
\end{enumerate}

\section{Framed quiver and super quiver varieties}
In this section, we will explore the super GIT quotient of super representation spaces (without doubling) that correspond to a given quiver and its framing. This allows us to construct a super quiver variety that parametrises the super representations of the quiver in the category of super vector spaces. 
We refer to \cite{CB, Nak} for more details on the framed quivers. 

Given a quiver $Q = (Q_0, Q_1, s, t)$, let's consider the framed quiver $Q^{\mathrm{fr}} = (Q^{\mathrm{fr}}_0, Q^{\mathrm{fr}}_1)$ by adding a single framing vertex, denoted $\infty$:
\begin{itemize}
    \item $Q^{\mathrm{fr}}_0 = Q_0 \cup \{\infty\}$.
    \\
    \item $Q^{\mathrm{fr}}_1 = Q_1 \cup \{a_i: i \to \infty \mid i \in Q_0'\}$\,
\end{itemize}
where $Q_0'$ is a subset of $Q_0$. Note that if $Q_0' = Q_0$, then the above notion of framed quiver coincides with that of \cite{CB}.
The vertices of the quiver $Q$ are called gauge vertices.

Given a super dimension vector $\sd$ for the quiver $Q$, we assign the super dimension $\hat{\sd} = (\sd, \d_\infty)$, where $\d_\infty$ is the super dimension assigned at the framing vertex $\infty$. 

The representation space $\bS\bfR(Q^{\mathrm{fr}}, \hat{\sd})$ is the super-vector space:
\[
\bS\bfR(Q^{\mathrm{fr}}, \hat{\sd}) = \left( \bigoplus_{a \in Q_1} \Hom_{k}(k^{\sd(s(a))}, k^{\sd(t(a))}) \right) \oplus \left( \bigoplus_{i \in Q_0'
} \Hom_{k}(k^{\sd(i)}, k^{\d_\infty}) \right)
\]
Let $\mathbb{V}_{\hat{\sd}}:= \Spec{\Sym(\bS\bfR(Q^{\mathrm{fr}}, \hat{\sd})^*)}$ be the corresponding affine super scheme.

By freezing the action at the framing vertex $\infty$, we can consider the action of the group 
$
\mathrm{G}(\sd) = \prod_{v\in Q_0} \mathrm{GL}(\d_v^0) \times \mathrm{GL}(\d_v^1)
$
on $\bS\bfR(Q^{\mathrm{fr}}, \hat{\sd})$ as follows:
\begin{itemize}
    \item For an internal arrow $a: i \to j$: $g \cdot \phi_a = g_j \phi_a g_i^{-1}$.
    \item For a framing arrow $a_i: i \to \infty $: the map $\phi_{a_i}$: $g \cdot \phi_{a_i} = \phi_{a_i}{g_i}^{-1}$.
\end{itemize}

Note that a point $x\in \mathbb{V}_{\hat{\sd}}(k)$ corresponds to the super representation of $Q^{\mathrm{fr}}$ having super dimension vector $\hat{\sd}$ in the category $\smod$ of super vector spaces.

Choose a stability parameter $\btheta \in (\Z\times \Z)^{Q_0}$, and consider the corresponding character $\chi_{\btheta} \colon G_{\sd} \to \mathbb{G}_m$.

To relate geometric stability with the King's algebraic framework, we define
\[
\hat{\btheta}(i) := \btheta(i) \quad (\forall i \in Q_0) \quad \text{and} \quad \hat{\btheta}(\infty) := \Big(-\frac{1}{\d_\infty^0}\sum_{i \in Q_0} \btheta(i) \cdot \sd(i) \;|\; 0\Big)
\]

Let $M\in \mathsf{Rep}(Q^{\mathrm{fr}}, \smod)$. We say that a super subrepresentation $M'$ of a framed super representation $M$ is \textbf{admissible} if $M'_\infty = 0\; \mathrm{or}\; M_\infty$.

\begin{definition}\label{def-frame-stability}
We say that a framed super representation $M$ is $\hat{\btheta}$-semistable if for every proper, non-zero admissible subrepresentation $M' \subset M$:
\[
\hat{\btheta}(M') = \Big( \sum_{i \in Q_0} \btheta_i \cdot \sdim(M'_i)\Big) + \hat{\btheta}_\infty \cdot \sdim(M'_\infty) \ge 0\,.
\]
\end{definition}

Then, we have the following:
\begin{proposition}
A point $x\in \mathbb{V}_{\hat{\sd}}(k)$ is $\chi_{\btheta}$-semistable
if and only if the corresponding framed super representation $M_x$ is $\hat{\btheta}$-semistable in the sense of Definition \ref{def-frame-stability}. 
\end{proposition}
\begin{proof}
Using results of \S 4 \& \S 5, for a representation $M$ of $Q^{\mathrm{fr}}$ in the category $\smod$, we have the following:
\begin{itemize}
    \item A point $x$ is $\chi_{\btheta}$-semistable if and only if $\langle \chi_{\btheta}, \lambda \rangle \ge 0$ for all 1-PS $\lambda$ where the limit exists.
    \item A point $x$ is $\chi_{\btheta}$-unstable if and only if there exists a 1-PS $\lambda$ where the limit exists such that $\langle \chi_{\btheta}, \lambda \rangle < 0$.
\end{itemize}
For a point $x \in \mathbb{V}_{\hat{\sd}}(k)$, let $M_x$ be the corresponding super representation.
Assume that $x$ is $\chi_{\btheta}$-semistable. 

\noindent \textbf{Case 1} ($\d'^{0}_{\infty} = 0$): Let $M' \subset M_x$ be a super subrepresentation with $\sdim(M')=(\sdp(v), \mathbf{0})$. Then, for the corresponding one-parameter subgroup $\lambda$, we have
$$
\langle \chi_{\btheta}, \lambda \rangle = \btheta\cdot \sdp = \sum_{i\in Q_0} \btheta(i)\cdot \sdp(i) = \hat{\btheta}(M')
$$
Since $x$ is $\chi_\theta$-semistable, we have $\hat{\btheta}(M') \ge 0$.

\noindent \textbf{Case 2} ($\d'^{0}_{\infty} = \d^{0}_{\infty}$): Let $M' \subset M_x$ be a super subrepresentation with $\sdim(M')=(\sdp(v), \d^{0}_{\infty})$.
Consider the one-parameter subgroup $\lambda$ which acts as identity on $M'_i$ and as multiplication by $t^{-1}$ on its complement in $M_i$. It is easy to see that the 
$\lim_{t \to 0} \lambda(t) \cdot M_x$ exists. Moreover,
\[
\chi_{\btheta}(\lambda(t)) = \prod_{i \in Q_0} (\det \lambda_i(t))^{\btheta(i)} = \prod_{i \in Q_0} (t^{-(\sd(i) - \sdp(i)})^{\btheta(i)} = t^{-\btheta \cdot (\sd - \sdp)} = t^{\btheta \cdot \sdp - \btheta \cdot \sd}
\]
Therefore, we have 
\[
\begin{aligned}
\langle \chi_{\btheta}, \lambda \rangle 
& = \btheta\cdot \sdp - \btheta \cdot \sd \\
& = \sum_{i\in Q_0} \btheta(i)\cdot \sdp(i) - \sum_{i \in Q_0} \btheta(i) \cdot \sd(i) \\
& = \hat{\btheta}(M')
\end{aligned}
\]
Since $x$ is $\chi_\theta$-semistable, we have $\hat{\btheta}(M') \ge 0$. Hence, we proved that $M_x$ is $\hat{\btheta}$-semistable in the sense of Definition \ref{def-frame-stability}.

Conversely, suppose that $M_x$ is $\hat{\btheta}$-semistable in the sense of Definition \ref{def-frame-stability}. Let $\lambda \colon \mathbb{G}_m \to G(\sd)$ be an arbitrary one-parameter subgroup such that $\lim_{t \to 0} \lambda(t) \cdot M_x$ exists. Then, we have
\[
\langle \chi_{\btheta}, \lambda \rangle 
= \sum_{n \in \mathbb{Z}} \btheta(M'_{n}) \\
\]
where $M'_n$ is an admissible super subrepresentation of $M_x$
determined by $\lambda$ (see Proposition \ref{prop-1-ps-filtration}, \ref{prop-king3.1}). Since $M_x$ is $\hat{\btheta}$-semistable in the sense of Definition \ref{def-frame-stability}, we have $\btheta(M'_{n}) \ge 0$ for all $n$. This shows that $\langle \chi_{\btheta}, \lambda \rangle \ge 0$. 
\end{proof}

\begin{remark}{\rm
The super GIT quotient $\mathbb{V}_{\hat{\sd}}/ \!\! /_{\! \!\chi_{\btheta}} G_\sd$ can be considered as the moduli of  $\hat{\btheta}$-semistable objects of $\mathsf{Rep}(Q^{\mathrm{fr}}, \smod)$ having super dimension vector $\hat{\sd}$.
}
\end{remark}

\subsection{Super flag varieties}
Consider a quiver $A_\ell$ and one framing vertex as follows: 
\[
\begin{tikzcd}
1 \arrow[r, "a_1"] & 2 \arrow[r, "a_2"] & \cdots \arrow[r, "a_{\ell-1}"] & \ell \arrow[r, "a_\ell"] & \infty
\end{tikzcd}
\]

To construct a flag of length $\ell$, we assign an ascending sequence of super-dimensions:

Let $\sd(i) = (r_i|s_i)$, where 
    \[ r_1 \leq r_2 \leq \dots \leq r_\ell \leq m \quad \text{and} \quad s_1 \leq s_2 \leq \dots \leq s_\ell \leq n \]
    such that either $r_i < r_{i+1}$ or $s_i < s_{i+1}$ for all $i = 1, 2, \dots \ell$; and let $\sd_\infty = (m|n)$. 

Then, we have the super-vector space of representations:

\[ \bS\bfR(A_\ell^{\mathrm{fr}}, \hat{\sd}) = \bigoplus_{i=1}^\ell \uHom_k(k^{r_i|s_i}, k^{r_{i+1}|s_{i+1}}) \]
where $(r_{\ell+1}|s_{\ell+1}) = (m|n)$.

Let $\mathbf{V}$ be the corresponding super affine scheme. Then,
a point in $\mathbf{V}(k)$ is represented by a sequence of $\ell$ supermatrices $\Phi = (\phi_1, \phi_2, \dots, \phi_\ell)$, where each $\phi_i$ is a even supermatrix.

If each $\phi_i$ is injective for all $i=1, 2, \dots, \ell$, 
by composing the maps, we get a strictly nested sequence of 
super-subspaces inside the fixed framing space $k^{m|n}$:
\[ V_1 \subsetneq V_2 \subsetneq \dots \subsetneq V_\ell \subseteq k^{m|n} \]
where $\sdim(V_i) = (r_i|s_i)$. 

Consider the action of the group 
\[ G(\sd) = \prod_{i=1}^\ell \text{GL}(r_i|s_i)_{\text{ev}} \cong \prod_{i=1}^\ell (\text{GL}(r_i) \times \text{GL}(s_i)) \]
on $\bS\bfR(A_\ell^{\mathrm{fr}}, \hat{\sd})$ as follows:
\[ g \cdot \phi_i = g_{i+1} \phi_i g_i^{-1} \]
where $g = (g_1, \dots, g_\ell) \in G$ and $g_{\ell+1} = \mathbf{I}_{m|n}$.

Consider the stability parameter $\hat{\btheta} = (\hat{\btheta}(i), \hat{\btheta}_\infty)$, where 
$$\hat{\btheta}(i) = (-m\;|\;-m) \quad \text{and} \quad \hat{\btheta}_\infty = \big(\sum_{i = 1}^\ell (r_i+s_i)\;|\;0\big)\,.$$

The corresponding character $\chi_{\btheta} \colon G(\sd) \ra \mathbb{G}_m$ is given by 
$$
g = (g_i) \mapsto \prod_{i=1}^\ell \det(g_i^0)^{\btheta_i^0}\cdot \det(g_i^1)^{\btheta_i^1}\,.
$$
Let $x\in \mathbf{V}(k)$, and $M_x$ be the corresponding representation of $A_\ell^\mathrm{fr}$ in $\mathsf{Rep}(A_\ell^\mathrm{fr}, \smod)$
having dimension vector $\hat{\sd}$. Note that $M_x$ can be represented as a sequence of $\ell$ supermatrices $\Phi_x = (\phi_1, \phi_2, \dots, \phi_\ell)$, where each $\phi_i$ is a even supermatrix of size $(r_{i+1}|s_{i+1})\times (r_i|s_i)$.
Then,
$$
\hat{\btheta}(M_x) = -m\big(\sum_{i=1}^\ell (r_i+s_i)\big) + m\big(\sum_{i=1}^\ell (r_i+s_i)\big) = 0\,.
$$
Let $M'$ be an admissible super subrepresentation of $M_x$ with 
$\sdim(M')=(\sd'(i), \sd'(\infty))$.
Now, 
$$
\begin{aligned}
\mbox{a point}\; x \;\mbox{is}\; \chi_{\btheta}-\mathrm{semistable} & \iff \hat{\btheta}(M') \ge 0 \\
& \iff \sum_{i=1}^\ell \hat{\btheta}(i) \cdot \sd'(i) + \hat{\btheta}_\infty \cdot \sd'(\infty) \ge 0 \\
& \iff -m\Big(\sum_{i=1}^\ell (\d'^{0}_{i} + \d'^{1}_{i})\Big) + \d'^{0}_{\infty}\Big(\sum_{i=1}^\ell (r_i+s_i)\Big) \geq 0\\
& \iff \d'^{0}_{\infty}\Big(\sum_{i=1}^\ell (r_i+s_i)\Big) \ge m\Big(\sum_{i=1}^\ell \d'^{0}_{i} + \d'^{1}_{i}\Big)
\end{aligned}
$$ 

If any $\phi_i \colon M_x(i)\ra M_x(i+1)$ has non-trivial kernel, then choosing $M'_i \subseteq \ker(\phi_i)$ with $\d'^{0}_{i} + \d'^{1}_{i} > 0$ and $\d'^{0}_{\infty} = 0$, we have $m\Big(\sum_{i=1}^\ell \d'^{0}_{i} + \d'^{1}_{i}\Big) > 0$. Hence, for a point $x$ to be $\chi_{\btheta}$-semistable, the corresponding maps in the super representation have to be injective.
Further, for a super representation $M_x$ with all $\phi_i$ injective maps, we always have $\hat{\btheta}(M') \ge 0$ for an admissible super subrepresentation $M'$ of $M_x$.
Therefore, we get
\[ 
\begin{aligned}
\mathbf{V}^{\rm ss}(k) & = \{ \Phi = (\phi_1, \dots, \phi_\ell)\in \bigoplus_{i=1}^\ell \underline{\Hom}_0(k^{r_i|s_i}, k^{r_{i+1}|s_{i+1}})\mid \phi_i \text{ is injective for all } i \}\,.\\
& = \mathbf{V}^{\rm st}(k)
\end{aligned}
\]
and hence
$$
\mathbf{V}_{\hat{\sd}}^{\mathrm{ss}}/ \!\! /_{\! \!\chi_{\btheta}} G(\sd) \cong \mathbf{SFlag}((r_1|s_1), (r_2|s_2), \dots (r_\ell|s_\ell); m|n)\,.
$$

\begin{remark}{\rm
If we take $\ell=1$, then we get 
$$
\mathbf{V}_{\hat{\sd}}^{\mathrm{ss}}/ \!\! /_{\! \!\chi_{\btheta}} G(\sd) \cong \mathbf{SGrass}(r_1|s_1; m|n)\,.
$$
Further, in this case, if we take $r_1 = 1$ and $s_1 = 0$, then we get 
$\mathbf{V}_{\hat{\sd}}^{\mathrm{ss}}/ \!\! /_{\! \!\chi_{\btheta}} G(\sd) \cong \mathbb{P}^{(m-1)\,|\, n}$.
For more details on super Grassmannians and super Flag varieties, we refer to \cite{BR21}.
} 
\end{remark}

\section*{Acknowledgements}
The first name author acknowledge the financial support provided by the Anusandhan National Research Foundation (ANRF), erstwhile Science and Engineering Research Board (SERB), Government of India, under the Core Research Grant scheme (File No. CRG/2023/000477). The second name author is partially supported by ANRF under project no. ANRF/ARGM/2025/002470/MTR and HRI, Prayagraj. The authors would like to thank Qiyuan (Alex) Gu for pointing out an inaccuracy in the previous version of Corollary 5.6. The authors are also grateful to Luis Alvarez-Consul, Tomas Gomez and Daniel Hernandez Ruiperez for highlighting the need for a precise definition of ``isomorphism classes'' of families in the context of the moduli functor in Section 5 of an earlier draft.



\begin{thebibliography}{012345}

\bibitem{BR21} Ugo Bruzzo, D. H. Ruiperez, \emph{The supermoduli of SUSY curves with Ramond punctures}, Rev. R. Acad. Cienc. Exactas Fís. Nat., Ser. A Mat. 115 (3) (2021) 144.

\bibitem{BRP} U.~Bruzzo,~D.~H.~Ruiperez,~A.~Polishchuk, \emph{Notes on fundamental algebraic supergeometry. Hilbert and Picard superschemes}, Adv. Math. 415 (2023) 115 pp.
   
\bibitem{BZ20}  V. A. Bovdi, A. N. Zubkov, \emph{Super-representations of quivers and related polynomial semi-invariants}. Internat. J. Algebra Comput. \textbf{30} (2020), 883-902.   

\bibitem{CB} W. Crawley-Boevey, \emph{Geometry of the moment map for representations of quivers}, Compositio Math. 126 (2001), 257--293.

\bibitem{CN} S.~L.~Cacciatori and S.~Noja, \emph{Projective superspaces in practice}, J. Geom. Phys., 130 (2018), pp. 40-62.

\bibitem{CCF} C. Carmeli, L. Caston, and R. Fioresi, \emph{Mathematical foundations of supersymmetry}, EMS Series of
Lectures in Mathematics, European Mathematical Society (EMS), Z\"urich, 2011.
        
\bibitem{DW15} Ron Donagi, Edward Witten, \emph{Supermoduli space is not projected}, Proc. Symp. Pure Math. 90 (2015) 19-71.

\bibitem{Ki94} A.~King, \emph{Moduli of representations of finite dimensional algebras}, Quart. J. Math. Oxford Ser. \textbf{45} (1994), 180, 515-530.
         
\bibitem{Ma88a} Y. I. Manin, \emph{Gauge field theory and complex geometry}, Grundlehren der Mathematischen Wissenschaften, vol. 289, Springer, 2013. 

\bibitem{MZ22} Masuoka, A.; Zubkov, A. N. \emph{Group superschemes}. J. Algebra 605 (2022), 89-145.

\bibitem{Nak} H. Nakajima, \emph{Varieties associated with quivers} in Representation Theory of Algebras and Related Topics (Mexico City, 1994), CMS Conf. Proc. 19, Amer. Math. Soc., Providence, 1996, 139--157.

\bibitem{OV23} Nadia Ott, Alexander A. Voronov, \emph{The supermoduli space of genus zero super Riemann surfaces with Ramond punctures}, J. Geom. Phys. 185 (2023) 104726.

\bibitem{S21} A. Sherman. \emph{Two geometric proofs of the classification of algebraic supergroups with semisimple representation theory}. 2021. \url{arXiv: 2012.11317}

\bibitem{We09} D. B. Westra. \emph{Superrings and Supergroups}. Ph.D. Thesis. Universitat Wien, Wien, 2009.
 		
\end{thebibliography}
\end{document}